\documentclass[conference]{ieeeconf}
\IEEEoverridecommandlockouts
\usepackage[noadjust]{cite}
\usepackage{amsmath,amssymb,amsfonts}
\usepackage{algorithmic}
\usepackage{graphicx}
\usepackage{textcomp}
\usepackage{xcolor}
\usepackage{dsfont}
\usepackage{amsthm}
\usepackage{xcolor}
\usepackage{subcaption}
\usepackage{dsfont}
\usepackage{textcomp}
\usepackage[hidelinks]{hyperref}
\let \labelindent \relax
\usepackage{enumitem}
\usepackage{accents}

\DeclareMathOperator{\sign}{sign}
\DeclareMathOperator{\spant}{span}
\DeclareMathOperator{\eig}{eig}

\newtheorem{theorem}{Theorem}[section]
\newtheorem{corollary}{Corollary}[theorem]
\newtheorem{lemma}[theorem]{Lemma}
\newtheorem{remark}{Remark}[section]
\newtheorem{proposition}[theorem]{Proposition}

\newcommand{\bs}[1]{\boldsymbol{#1}}

\newcommand{\Na}{{N_{\rm a}}}
\newcommand{\No}{{N_{\rm o}}}

\newcommand{\Zz}{\bs{Z}}

\newcommand{\Aa}{A_{\rm a}}

\newcommand{\range}[1]{\operatorname{range} #1}

\definecolor{mygreen}{rgb}{0.020, 0.750, 0.480} % RGB values between 0 and 1

\def\BibTeX{{\rm B\kern-.05em{\sc i\kern-.025em b}\kern-.08em
    T\kern-.1667em\lower.7ex\hbox{E}\kern-.125emX}}

\begin{document}

\title{Constraint-Induced Redistribution of Social Influence in Nonlinear Opinion Dynamics
\thanks{}
}

\author{Vishnudatta~Thota, Anastasia~Bizyaeva% <-this % stops a space
\thanks{V. Thota and A. Bizyaeva are with the Sibley School of Mechanical and Aerospace Engineering at Cornell University, Ithaca, NY, 14850; \tt{\{vt279, anastasiab\}@cornell.edu} 
}
}%

\maketitle

\begin{abstract}
We study how intrinsic hard constraints on the decision dynamics of social agents shape collective decisions on multiple alternatives in a heterogeneous group. Such constraints may arise due to structural and behavioral limitations, such as adherence to belief systems in social networks or hardware limitations in autonomous networks. In this work, agent constraints are encoded as projections in a multi-alternative nonlinear opinion dynamics framework. We prove that projections induce an invariant subspace on which the constraints are always satisfied and study the dynamics of networked opinions on this subspace. We then show that heterogeneous pairwise alignments between individuals' constraint vectors generate an effective weighted social graph on the invariant subspace, even when agents exchange opinions over an unweighted communication graph in practice. With analysis and simulation studies, we illustrate how the effective constraint-induced weighted graph reshapes the centrality of agents in the decision process and the group's sensitivity to distributed inputs.

\end{abstract}

%\begin{IEEEkeywords}
%\end{IEEEkeywords}

\section{Introduction}

Collective decision-making is an important component of multi-agent autonomy across application domains, from navigation in multi-vehicle networks \cite{ren2007distributed, rios2016automated}
to task allocation in multi-robot teams \cite{gerkey2004formal, khamis2015multi}.
Analogous collective decision processes arise in natural social systems, where individuals form and update opinions about different options through structured social interactions \cite{degroot1974reaching,seeley1999group,jia2015opinion}.
In each of these settings, the collective decision outcome can depend not only on how the agents communicate, but also on how their individual constraints shape the information they exchange and act upon. 
In this paper, we study the interplay between heterogeneous agent constraints and outcomes of collective decisions. Principled mathematical understanding of this interplay is important both for analyzing natural social systems and for designing autonomous behaviors.

Mathematical models of networked opinion dynamics, e.g., \cite{degroot1974reaching,altafini2012consensus, jia2015opinion, fontan2018achieving,he2023opinion}, are used to study social interactions in multi-agent systems and to design algorithms for multi-agent autonomy. Among these, the Nonlinear Opinion Dynamics (NOD) model \cite{bizyaeva2023tac,bizyaeva2025tac,leonard2024fast} is a tractable framework for tunably sensitive collective decision-making on multiple alternatives. In NOD, collective decisions emerge through a bifurcation that yields multistable agreement and disagreement outcomes even for homogeneous, i.e., identical, agents, and has primarily been studied under such homogeneity assumptions.
The NOD framework is also emerging as a tool for the design of flexible cooperative autonomy across domains \cite{cathcart2023proactive,paine2024model,martinez2025avocado,qi2025integrating}. In this paper, we build on NOD to characterize the effects of agent constraints and group heterogeneity on the outcome of multi-alternative collective decisions.

Agent heterogeneities, such as hardware or resource limitations in autonomous robots or adherence to heterogeneous belief systems in social networks,  can shape both how they interact and which opinion states are feasible. Existing work has studied the effects of heterogeneous interaction weights, confidence bounds, and inter-option coupling on opinion formation in various frameworks, from linear consensus dynamics to bounded confidence models \cite{mirtabatabaei2012opinion,liang2013opinion,li2023bounded,he2023opinion,zhang2026modeling}. 
A related line of work considers constrained opinion dynamics, including limiting the agents' linear opinion updates to constraint sets \cite{nedic2010constrained}, resource allocation formulations with projected linear opinion dynamics \cite{wankhede2025multi}, and opinion dynamics defined on curved manifolds such as the unit sphere \cite{zhang2021dissensus,zhang2022opinion,zhang2024opinion,zhang2025mixed}. However, such constraints have not yet been considered within the NOD framework \cite{bizyaeva2023tac}. 

In this paper, our contributions are as follows. First, we introduce heterogeneous, agent-specific hard constraints into the NOD framework. 
Analogous to the work on projected linear consensus flows \cite{nedic2010constrained,wankhede2025multi}, these constraints are enforced via local projections in the agents' opinion update dynamics. We prove that local projection constraints induce a global invariant subspace on which individual agents' constraints are always enforced and derive a reduced representation of the dynamics of networked opinions on this subspace. 
Second, we analyze the reduced model with one-dimensional constraints and characterize the emergence of constrained network decision states via a pitchfork bifurcation and its unfolding, analogously to the unconstrained NOD. 
These one-dimensional constraints are analytically tractable and provide a simple way to encode the relative importance of different options for an agent.
Third, we prove that heterogeneous pairwise alignments between agents' constraint vectors generate an effective weighted social graph, even when the true graph is unweighted. We illustrate how this constraint-induced weighted graph redistributes social influence and, together with local input projection, shapes the group's sensitivity to distributed inputs.

The paper is organized as follows. Section \ref{sec:Notation and Definitions} contains preliminaries from matrix theory and graph theory. In Section \ref{sec:Projection-Constrained NOD} we introduce and analyze projection-constrained NOD and characterize the emergent influence redistribution in the presence of heterogeneous constraints for several graph types. In Section \ref{sec:Simulations} we present illustrative numerical simulations. Finally, in Section \ref{sec:Conclusion} we conclude the paper with directions for future research.

\section{Notation and Preliminaries}
\label{sec:Notation and Definitions}
Let $\mathbf{1}_{n}$ and $\mathbf{0}_{n}$ denote the $n-$dimensional column vectors of ones and zeros. The matrix $I$ denotes the identity matrix of appropriate dimensions. Let $M \in \mathbb{R}^{n\times m}$. We use $\ker(M)$ and $\range(M)$ to represent the kernel of $M$ and range of $M$. $M$ is said to be irreducible if it is not similar to a block upper triangular matrix via a permutation matrix. 
Let $\sign$ denote the sign function, with $\sign(0) = 0$. 
We use $C^{k}$ to denote the class of $k$-times continuously differentiable functions.
A vector $\mathbf{v} \in \mathbb{R}^{n}$ is called strictly positive if all its entries are positive and it is represented by $\mathbf{v} > 0$. 
For $M$ with full column rank, the matrix $P = M (M^T M)^{-1} M^T \in \mathbb{R}^{n \times n}$ defines an orthogonal projection onto the range of $M$ and satisfies $P^2 = P$, with complementary projection $P^{\perp} = I - P$. 

A weighted, signed graph $\mathcal{G} = (\mathcal{V},\mathcal{E},A)$ consists of a node set $\mathcal{V} = \{1, \dots, n\}$, an edge set $\mathcal{E}\subseteq \mathcal{V} \times \mathcal{V}$, and a weighted adjacency matrix $A \in \mathbb{R}^{n \times n}$ with entries $a_{ij} \neq 0$ if $(i,j) \in \mathcal{E}$ and $a_{ij} = 0$ otherwise. We consider simple graphs, i.e., ones with no self-loops, $a_{ii} = 0$ for all $i \in \mathcal{V}$. A graph is undirected if $(i,j) \in \mathcal{E}$ if and only if $(j,i) \in \mathcal{E}$, and $A = A^T$. A graph is connected if there exists a path between every pair of distinct nodes. Equivalently, $A$ is irreducible. A graph is unweighted if $a_{ij} \in \{0,1\}$ for all $i,j \in\mathcal{V}$. We say the interaction between nodes $i,j$ is cooperative when $a_{ij} > 0$ and antagonistic when $a_{ij} < 0$.
A graph is unsigned if $a_{ij} \geq 0$ for all $(i,j) \in \mathcal{E}$, and signed if some edges have negative weight. 

For an undirected, unweighted d-regular network, each node $i \in \mathcal{V}$ has degree $d = \sum_{j=1}^{n}a_{ij}$. In this case, $A\mathbf{1}_{n} = d\mathbf{1}_{n}$, and hence the dominant eigenvalue of the $A$ is $\lambda_{\max}= d$ and the corresponding eigenvector is $\mathbf{v} = \mathbf{1}_{n}$. For an undirected, unweighted star network, let node 1 denote the central node. The dominant eigenvalue of $A$ is $\lambda_{\max}= \sqrt{n-1}$, and the corresponding eigenvector is $\mathbf{v}=\left(\sqrt{n-1}, \mathbf{1}_{n-1}^{T} \right)^{T}$. More generally, if a network is unsigned and connected, then it follows from the Perron-Frobenius Theorem \cite[Theorem 2.12]{bullo2018lectures} that the dominant eigenvalue $\lambda_{\max}$ of $A$ is simple, and the corresponding eigenvector $\mathbf{v}$ is strictly positive and unique up to scaling. This eigenvector $\mathbf{v}$ defines the eigenvector centrality of the network, with each component $v_{i}$ quantifying the relative influence of node $i$.

\begin{lemma}[\cite{fox1968rates}]
\label{lemma:eigen_vecotr_value_rate}
Consider the eigenvalue problem, $K\mathbf{v}_{i} = \lambda_{i} \mathbf{v}_{i}$, where $\mathbf{v}_{i}$ is a unit eigenvector corresponding to a simple eigenvalue $\lambda_{i}$ and $K \in \mathbb{R}^{n \times n}$. Let $K$ be a symmetric matrix, and let $K\left(\delta\right) \in \mathbb{R}^{n \times n}$ be the symmetric matrix obtained after adding $\delta$ to some entries of $K$. To first order, the eigenpair under the perturbation $\delta$ is
\begin{equation}
\label{eq:eigenpair_approximation}
\lambda_{i}^{*}  = \lambda_{i} + \delta \frac{\partial \lambda_{i}}{\partial \delta} + \mathcal{O}(\delta^2), \qquad
\mathbf{v}_{i}^{*} = \mathbf{v}_{i} + \delta \frac{\partial \mathbf{v}_{i}}{\partial \delta} + \mathcal{O}(\delta^2)
\end{equation}
where $\frac{\partial \lambda_{i}}{\partial \delta} = \mathbf{v}_{i}^{T}K'\left(\delta\right)\mathbf{v}_{i}$,  $\frac{\partial \mathbf{v}_{i}}{\partial \delta} = - \left[F_{i}\left(0\right)F_{i}\left(0\right) +  2\mathbf{v}_{i}\mathbf{v}_{i}^{T} \right]^{-1} F_{i}\left(0\right)F_{i}'\left(\delta\right)\mathbf{v}_{i}$ and $F_{i}\left(\delta\right) = K\left(\delta\right) - \lambda_{i}I$ . Here $K'\left(\delta\right)$ and $F_{i}^{'}\left(\delta\right)$ denote the matrices formed by differentiating the elements of $K\left(\delta\right)$ and $F_{i}\left(\delta\right)$ matrices, respectively, with respect to $\delta$. 
\end{lemma}

\section{Projection-Constrained NOD}
\label{sec:Projection-Constrained NOD}
We consider $\Na$ agents forming and exchanging opinions on $\No$ options over an undirected and connected communication network with a static topology. 
Agent interactions are encoded in graph $\mathcal{G} = \left(\mathcal{V}, \mathcal{E}, \Aa\right)$, where $\mathcal{V}$ is the node set, $\mathcal{E}$ is the edge set and $\Aa \in \mathbb{R}^{\Na \times \Na}$ is the inter-agent adjacency matrix. Here, $\Aa$ is symmetric, irreducible, with entries $a_{ij} \in \{0,1\}$ and $a_{ii} = 0$ for all $i \in \mathcal{V}$. 

Let $z_{ij} \in \mathbb{R}$ be the opinion of agent $i$ on option $j$, with $z_{ij} > 0 (< 0)$ representing a preference (rejection) of option $j$ and $|z_{ij}|$ reflecting strength of commitment to this decision. We say $z_{ij} = 0$ means agent $i$ is neutral on option $j$.  
Each agent $i$ has a set of constraints on its opinions, for example, arising from a learned belief system if it is an agent in a social network or a set of hardware constraints if opinions represent allocation of onboard resources to tasks. We assume that each set of opinion constraints of agent $i$ can be encoded in an orthogonal projection matrix $P_{i}\in \mathbb{R}^{\No \times \No}$ whose complementary projection is $P_{i}^{\perp} = I - P_{i}$. 
The opinions of agent $i$ are represented by a vector $\Zz_{i} = \left(z_{i1},\cdots,z_{i\No}\right)^{T} \in V_{P_i} \subseteq \mathbb{R}^{\No}$, where $V_{P_i} = \{ \mathbf{x} \in \mathbb{R}^{\No}  \ s.t. \ P_i^{\perp} \mathbf{x} = 0 \}$. The network state $\Zz = \left(\Zz_{1}^{T}, \cdots, \Zz_{\Na}^{T}\right)^{T} \in V_{P_1} \times V_{P_2} \times \dots \times V_{P_{\Na}} \subseteq \mathbb{R}^{\Na \No}$ represents the opinions of all the agents.

Each agent $i$ updates its opinion according to %the nonlinear update rule,
\begin{subequations} \label{eq:CNOD-simple}
\begin{align}
    \dot{\Zz}_{i} &= P_i \mathbf{F}_{i}(\Zz), \label{eq:CNOD_simple_proj} \\
    F_{ij}(\Zz) &= -d z_{ij} + S\left( u \left(\alpha z_{ij}  + \gamma \, \sum_{\substack{k=1 \\ k \neq i}}^\Na(A_{\rm a})_{ik} z_{kj} \right) \right)+ b_{ij} \label{eq:CNOD_simple_social} %\nonumber
\end{align}
\end{subequations}
where $\mathbf{F}_i(\Zz) = (F_{i1}(\Zz), \dots, F_{i \No}(\Zz) )$ and $S:\mathbb{R} \to \mathbb{R}$ an odd sigmoidal saturating function which satisfies $S(0) = 0$, $S'(0) = 1$ and $\sign (S''(z)) = -\sign(z)$. 
Model parameters include attention ($u >0$), damping ($d>0$), social influence ($\gamma >0$), strength of self-reinforcement ($\alpha >0$), and external input or intrinsic bias of agent $i$ on option $j$ ($b_{ij} \in \mathbb{R}$) .

The saturating influence of neighbors on each agent's opinion update in \eqref{eq:CNOD-simple} follows the general form of multidimensional Nonlinear Opinion Dynamics models \cite{bizyaeva2023tac,bizyaeva2025tac}. 
The model in \cite{bizyaeva2023tac} combines all-to-all inter-option coupling with a shared projection $P_i = I - \frac{1}{N_o} \mathbf{1} \mathbf{1}^T$ onto the zero-sum opinion subspace, whereas in \cite{bizyaeva2025tac} inter-option coupling is graph based, introducing a soft constraint that indirectly influences opinion evolution.
Here we consider opinion coupling through hard constraints encoded by the agent-specific projection matrices $P_{i}$.
Importantly, these constraints can be heterogeneous within the group, and the focus of our analysis is how this heterogeneity shapes the collective decision and sensitivity to inputs.
In the following proposition, we prove that %the hard constraint 
$P^{\perp}_{i} \Zz_i(t) = 0$ is enforced along trajectories of \eqref{eq:CNOD-simple} on the invariant subspace $V_{P_1}\times V_{P_2} \times \dots V_{P_{\Na}}$.
\begin{proposition}[Constraint Enforcement]\label{prop:constraint_enforcement}
    Consider \eqref{eq:CNOD-simple} and suppose $P_{i}$ has rank $k \geq 1$.  Then $\range{P_{i}} = \ker{P^{\perp}_{i}}$ is invariant under the flow of \eqref{eq:CNOD-simple}.
\end{proposition}
\begin{proof}
    Suppose $\Zz_i(0) \in \ker{P^{\perp}_{i}}$. Then $P^{\perp}_{i} \Zz_i(0) = 0$. The time derivative $\frac{d}{dt} (P^{\perp}_{i} \Zz_i) = P^{\perp}_{i} \dot{\Zz}_i = P^{\perp}_{i} P_{i} \mathbf{F}_i(\Zz) = 0$. Therefore $\Zz_i(t) \in \ker{P^{\perp}_{i}}$ for all $t \geq 0$.
\end{proof}
\begin{remark}
Each agent's projection constraint is encoded locally in its own update rule \eqref{eq:CNOD-simple}, 
so constraint enforcement does not require agents to explicitly communicate their constraints to their neighbors. 
This is a useful feature of the opinion dynamics model for settings where constraints may reflect private, agent-specific properties such as hardware limitations that may not be directly observable or secure to share over a communication network.
\end{remark}

\subsection{Bifurcation analysis under one-dimensional constraints}
We consider rank-one projection constraints, in which each agent $i$'s constraint subspace is one-dimensional, spanned by a single vector $\mathbf{p}_i \in \mathbb{R}^{\No}$.
We define the matrix $P_i$ to be a projection matrix onto the span of $\mathbf{p}_i$ in $\mathbb{R}^\No$, 
\begin{equation}
    P_i = \frac{1}{\|\mathbf{p}_i\|^{2}}\mathbf{p}_i \mathbf{p}_{i}^T = \hat{\mathbf{p}}_i \hat{\mathbf{p}}_{i}^T. \label{eq:proj_matrix_onto_span_of_p}
\end{equation}
where $\hat{\mathbf{p}}_i = \frac{\mathbf{p}_{i}}{\|\mathbf{p}_{i}\|} $ and the norm $\|\mathbf{p}_{i}\|$ is the standard Euclidean L2 norm.
We will refer to $\mathbf{p}_{i}$ and $\hat{\mathbf{p}}_{i}$ as a \textit{constraint vector} and the span of $\mathbf{p}_{i}$ (equivalently $\hat{\mathbf{p}}_{i}$) as the \textit{constraint subspace} of agent $i$.
By Proposition \ref{prop:constraint_enforcement}, the flow of \eqref{eq:CNOD-simple} is invariant on $V_{P_i} = \operatorname{span}\{\mathbf{p}_i\}$ for each agent $i$. Therefore, on this invariant subspace, the opinion vector of agent $i$ satisfies $\Zz_i = y_i \hat{\mathbf{p}}_i$, where $y_i = \hat{\mathbf{p}}_i^T \Zz_i \in \mathbb{R}$ is the \textit{effective opinion} of agent $i$. In the following proposition, we illustrate that, restricted to the constraint subspace, the $\Na\No$-dimensional dynamics of \eqref{eq:CNOD-simple} reduce to an $\Na$-dimensional system in the effective opinions. For every agent $i$, let $y_i \in \mathbb{R}$ and $b_{ei} \in \mathbb{R}$ denote the projections of $\Zz_{i}$ and $\mathbf{b}_{i} =\left(b_{i1},\hdots,b_{i\No}\right)^{T}$, respectively, onto the constraint vector $\mathbf{p}_{i}$. 

\begin{proposition}[Reduced Dynamics]\label{prop:reduced_dyn}
Consider \eqref{eq:CNOD-simple} restricted to the constraint subspace $V_{P_1}\times V_{P_2} \times \dots V_{P_{\Na}}$ induced by rank-one projection constraints \eqref{eq:proj_matrix_onto_span_of_p} for each agent $i \in \mathcal{V}$. The flow is exactly characterized by dynamics of the agent's effective opinion vector $\mathbf{y} =\left(y_{1},\hdots,y_{\Na}\right)^{T} \in \mathbb{R}^{\Na}$ with effective bias $\mathbf{b}_{e} =\left(b_{e1},\hdots,b_{e\Na}\right)^{T}\in \mathbb{R}^{\Na}$, which evolves as $\dot{\mathbf{y}} = \Phi \left(\mathbf{y},u \right)$ where

\begin{align}
\label{eq:CNOD-multi-agnet-y-domain}
\dot{y}_{i} & = \Phi_{i}\left(\mathbf{y}, u\right) \\
~=&-dy_{i} + \hat{\mathbf{p}}_i^{T} S\left(u\left(\alpha y_{i}\hat{\mathbf{p}}_i+ \gamma \, \sum_{\substack{k=1 \\ k \neq i}}^\Na(A_{\rm a})_{ik} y_{k}\hat{\mathbf{p}}_k\right)\right) + b_{ei},   \nonumber
\end{align}
$y_{i} = \hat{\mathbf{p}}_i^T \Zz_{i}$, and $b_{ei} = \hat{\mathbf{p}}_i^T \mathbf{b}_{i}$.
\end{proposition}
\begin{proof}
    Differentiating $y_i = \hat{p}_i^T \Zz_i$ along trajectories of \eqref{eq:CNOD-simple} yields $\dot{y}_i = \hat{p}_i^T \dot{\Zz}_i = \hat{\mathbf{p}}_i^T P_i \mathbf{F}_i(\Zz) =\hat{\mathbf{p}}_i^T \mathbf{F}_i(\Zz) $. The expression then follows by direct substitution of $F_{ij}(\Zz)$ and substitution of $\Zz_i = y_i \hat{\mathbf{p}}_i$, $b_{ei}$.
\end{proof}

Observe that the effective bias $b_{ei}$ of agent $i$ in Proposition \ref{prop:reduced_dyn} depends on the projection of the bias of its individual options $\mathbf{b}_i$ along the constraint unit vector $\hat{\mathbf{p}}_i$. 
If an entry in the constraint vector is zero, the corresponding option bias does not affect the effective bias, regardless of its magnitude.
Hence, projection constraints enable the agents to be selectively sensitive to the option biases.
For example, if opinions are mapped to tasks an agent must execute or prioritize, an aerial drone may have some zero entries in the projection constraint vector corresponding to the options for ground tasks; therefore, even strong incentives for ground tasks will not influence its effective decision.

To identify bifurcations from neutral equilibria, we evaluate the Jacobian of Eqn. \eqref{eq:CNOD-multi-agnet-y-domain} at $\mathbf{y} = \mathbf{0}$, which reads 
\begin{subequations}
\label{eq:Jacobian-y-domain}
\begin{align}
J_{y} & = (u\alpha-d)I + u\gamma A_{\rm a}', \label{eq:Jacobian-y-domain_a}\\ 
\left[A_{\rm a}'\right]_{ik} & =(\hat{\mathbf{p}}_i^{T}\hat{\mathbf{p}}_k) \left[A_{\rm a}\right]_{ik} \label{eq:Jacobian-y-domain_b}
\end{align}
\end{subequations}
where $\Aa' $ is the transformed adjacency matrix. 
From the structure of \eqref{eq:Jacobian-y-domain}, we observe that projection constraints induce to first order an effective communication network $\mathcal{G}'$ with adjacency matrix $\Aa'$. Even if the entries of $\Aa$ are unweighted (i.e., 0 or 1), its corresponding $\Aa'$ can be weighted, and even signed depending on the structure of the constraint vectors $\hat{\mathbf{p}}_i$. The effective communication weight between two agents depends on the alignment of their projection constraints.
Then if two agents have orthogonal constraint directions, the edge between them in the effective graph has zero weight even when they are connected in the communication graph. 
It is therefore an emergent property of the constrained opinion dynamics \eqref{eq:CNOD-simple}
that agents with orthogonal capabilities effectively ignore each other in their decision-making when network opinions are close to global neutrality.

In the following theorem, we classify the bifurcation of the unopinionated equilibrium $\mathbf{y}=\mathbf{0}$ arising from a simple eigenvalue of the adjacency matrix $\Aa'$ as a pitchfork bifurcation, analogous to results established for the unconstrained NOD frameworks \cite{bizyaeva2021control,bizyaeva2023tac,bizyaeva2025tac}.

\begin{theorem}[Pitchfork Bifurcation]
\label{thm:Multi_agnet_bifurcation}
Consider Eqn. \eqref{eq:CNOD-multi-agnet-y-domain} with $S \in C^3$ and $S'''(0) < 0$. Let $\lambda$ be a simple real eigenvalue of $\Aa'$ with unit eigenvector $\mathbf{v} = \left(v_{1},\hdots,v_{\Na}\right)^{T}$.
Define
$c = \frac{1}{\alpha + \lambda \gamma} \sum_{i,j} v_i [\hat{\mathbf{p}}_i]_j  \left(\alpha v_{i} \left[\hat{\mathbf{p}}_i\right]_{j} + \gamma \, \sum_{\substack{k=1 \\ k \neq i}}^\Na(A_{\rm a})_{ik} v_{k}\left[\hat{\mathbf{p}}_k\right]_{j}  \right)^3$ for $i \in \{1,\cdots, \Na\}$ and $j \in \{1,\cdots,\No\}$.
Assume that $\alpha + \lambda \gamma \neq 0$ and $c \neq 0$ and define $u^{*} = \frac{d}{\alpha + \lambda \gamma}$.
For zero bias projection $\left(\mathbf{b}_{e} = \mathbf{0}\right)$, $h(x, u)$ has a symmetric pitchfork singularity at $(x, u) = (0, u^{*})$. For sufficiently small $|u-u^{*}|$, two branches of equilibria branch off from $\mathbf{y}=\mathbf{0}$ in a pitchfork bifurcation along a manifold tangent
to  $\spant \{\mathbf{v}\}$.
When $(\alpha + \lambda \gamma) c > 0~\left(<0\right)$, the bifurcation happens supercritically (subcritically) with respect to $u$. 
For $\mathbf{b}_{e} \neq \mathbf{0}$, the bifurcation problem is an $\Na-$parameter unfolding of the symmetric pitchfork with $\frac{\partial h}{\partial b_{ei}}(0,u^*) = v_{i}$.
\end{theorem}
\begin{proof}
The eigenvalues of the Jacobian $J_{y}$ in Eqn. \eqref{eq:Jacobian-y-domain_a} are $\eig (J_{y}) = u\alpha-d + u\gamma \lambda$. So, $J_{y}$ has a single zero eigenvalue when $u=u^{*}$. The left and right eigenvector for this zero eigenvalue is $\mathbf{v}$. 
We follow the procedure outlined in \cite[Theorem IV.I]{bizyaeva2021control} to derive the following truncated series expansion of the Lyapunov-Schmidt reduction of Eqn. \eqref{eq:CNOD-multi-agnet-y-domain} about $(0, u^{*})$:
$h(x, u) = a\hat{u}x+ \frac{1}{6} bx^{3}+\mathbf{v}^{T}\mathbf{b}_{e}$
where $\hat{u} = u - u^{*}$.
The coefficients of expansion are $a = \alpha + \lambda \gamma$ and $b = cd\left(u^{*}\right)^{2}S'''(0)$.
Furthermore, $\frac{\partial h}{\partial b_{ei}} = v_{i}$. 
The zero-bias projection part of the proof follows from the recognition problem for the pitchfork bifurcation  \cite[Proposition II.9.2]{schaeffer1985singularities} applied to the reduced equation and the definition of center manifold. The unfolding of the pitchfork for $\left(\mathbf{b}_{e} \neq \mathbf{0}\right)$ follows from the unfolding theory \cite[Proposition III.4.4]{schaeffer1985singularities}.
\end{proof}

\subsection{Social Influence Redistribution}
It follows from Theorem \ref{thm:Multi_agnet_bifurcation} that social influence of agents can be quantified using the eigenvector $\mathbf{v}$ of the effective adjacency matrix $A_a'$ in \eqref{eq:Jacobian-y-domain_b} corresponding to its largest real part eigenvalue, assuming it is simple, since it will correspond to the collective decision in which neutrality loses stability. The agents' effective biases enter the reduced bifurcation equation as $\mathbf{v}^T \mathbf{b}_e = \sum_{i = 1}^{\Na} v_i b_{ei}$, and the sign of this term determines whether the unfolded pitchfork bifurcation diagram is biased towards a network decision aligned towards $\mathbf{v}$ or $-\mathbf{v}$ to first order in the reduced variable $\mathbf{y}$. Thus, $|v_i|$ determines the relative weight of agent $i$'s influence in shaping the collective decision and $\mathbf{v}$ acts as an effective eigenvector centrality for the collective decision. 
Since this effective centrality is shaped by the interplay between network topology and constraint alignment, agents with equal influence in $\mathcal{G}$ can have unequal influence in $\mathcal{G}'$ due to heterogeneous projection constraints. 
In this section, we illustrate this influence redistribution effect explicitly for several graph types by deriving the effective centrality.

\begin{theorem}[Effective Centrality on Star Graphs]
\label{thm:social_influence_redistribution_star_network}
Let $\mathcal{G}$ be an undirected, connected, star network with $N_a \geq 3$ nodes, and let $\mathcal{G}'$ be its effective graph defined by \eqref{eq:Jacobian-y-domain_b}. 
Suppose that $\hat{\mathbf{p}}_{1}^{T}.\hat{\mathbf{p}}_{j} \neq 0 $, where 1 is the central node and $j \in \{2,\cdots,\Na\}$ are the outer nodes.
Then the largest eigenvalue of $\Aa'$ is simple and its corresponding effective centrality eigenvector is $\mathbf{v}_{\max} = \left(K, \hat{\mathbf{p}}_{1}^{T}\hat{\mathbf{p}}_{2}, \hdots,  \hat{\mathbf{p}}_{1}^{T}\hat{\mathbf{p}}_{\Na}\right)^{T}$ where $K=\sqrt{\sum_{i=2}^{\Na} \left(\hat{\mathbf{p}}_{1}^{T}.\hat{\mathbf{p}}_i\right)^{2}}$.
Thus, the central node remains uniquely most influential, while the influence strengths of the outer nodes are proportional to their absolute constraint alignments with the central node.
\end{theorem}

\begin{proof}
By \cite[Lemma 2.6]{Hiesmayr2024}, the leading simple eigenvalue for a weighted star graph adjacency matrix \eqref{eq:Jacobian-y-domain_b} is $K$. Let $\mathbf{v}_{\max} = (v_1, \dots, v_{\Na})^T$ be the corresponding eigenvector. Then $A' \mathbf{v} = K \mathbf{v}$ implies $K v_j = (\hat{\mathbf{p}}_1^T \hat{\mathbf{p}}_j) v_1$ for $j = 2, \dots, N_a$. Choosing $v_1 = K$, we arrive at the stated expression.
\end{proof}

For a graph of homogeneous agents, by \ref{thm:social_influence_redistribution_star_network} the effective centrality is $\mathbf{v}_{\max} = \left(\sqrt{\Na-1}, \mathbf{1}_{\Na-1}^{T} \right)^{T}$. The central node is therefore maximally influential on a star graph with any arbitrary constraint distribution, but heterogeneous constraints reshape the relative influence of outer nodes.

\begin{theorem}[Effective Centrality on Regular Graphs with One Heterogeneous Node]
\label{thm:centrality_regular_graph_approx} 
Let $\mathcal{G}$ be an undirected, connected, $d$-regular graph over $N_a$ nodes, and let $\mathcal{G}'$ be its corresponding effective graph with adjacency matrix \eqref{eq:Jacobian-y-domain_b} given by $\hat{\mathbf{p}}_i = \hat{\mathbf{p}}$ for $i = \{2, \dots, N_a\}$ with $\hat{\mathbf{p}}_1 \neq \hat{\mathbf{p}}$. 
To first order in $\delta = \hat{\mathbf{p}}_1^T\hat{\mathbf{p}} - 1$ as $\delta \to 0$, the approximate eigenvector centrality of $\mathcal{G}'$ is
\begin{equation}
\label{eq:centrality_regular_graph_approx}
\mathbf{v}_{\max}' \approx \mathbf{1}_{\Na} + \left( 1 - \hat{\mathbf{p}}_{1}^{T}.\hat{\mathbf{p}}\right) B\left(-d^{2} + d, M^{T}\right)^{T}
\end{equation}
where $B =\frac{1}{2\Na}\mathbf{1}_{\Na}\mathbf{1}_{\Na}^{T} + \left(\left(\Aa-dI\right)^{2}\right)^{\dagger}$, and $M = \left(m_{12},\cdots,m_{1\Na}\right)^{T} \in \mathbb{R}^{\Na-1}$.
Here, $m_{1j} = \sum_{k=1}^{\Na}\left[A_{\rm a}\right]_{1k}\left[A_{\rm a}\right]_{kj} \in \mathbb{R}$ is the number of mutual neighbors of agent $1$ and agent $j$, and $\left(\left(\Aa-dI\right)^{2}\right)^{\dagger}$ is the Moore-Penrose pseudo-inverse of $\left(\Aa-dI\right)^{2}$.
\end{theorem}
\begin{proof}
The graph $\mathcal{G}$ is an unweighted d-regular network. The eigenvector centrality of $\Aa$ is $\mathbf{1}_{\Na}$. The interactions in $\mathcal{G}'$ are weighted due to the heterogeneous projection constraints. The interaction weights of $d$ edges involving agent $1$ is $\hat{\mathbf{p}}_{1}^{T}.\hat{\mathbf{p}}$ and all other edge weights are either 0 or 1. The eigenvector centrality of $\Aa'$ can be approximated using Lemma \ref{lemma:eigen_vecotr_value_rate}, with $\delta = \hat{\mathbf{p}}_{1}^{T}.\hat{\mathbf{p}} - 1$ representing the change in the interaction weights involving agent $1$ due to perturbation in the projection constraint of agent $1$. This gives $\lambda_{\max}' \approx d - \frac{2d}{\Na}\left(1 -\hat{\mathbf{p}}_{1}^{T}.\hat{\mathbf{p}} \right)$ and $\mathbf{v}_{\max}'$ in Eqn.~\eqref{eq:centrality_regular_graph_approx}.
\end{proof}

The following corollaries follow from Theorem \ref{thm:centrality_regular_graph_approx}:

\begin{corollary}[Effective Centrality on Complete Graphs with One Heterogeneous Node]
\label{corollary:eig_vec_approx_complete_graph}
Let $\mathcal{G}$ in Theorem \ref{thm:centrality_regular_graph_approx} be a complete graph. The effective eigenvector centrality is approximated to first order in $\delta = \left( \hat{\mathbf{p}}_{1}^{T}\hat{\mathbf{p}} -1\right)$  by 
\begin{equation}
\label{eq:centrality_complete_graph_approx}
\mathbf{v}_{\max}' \approx \mathbf{1}_{\Na} - \frac{\delta\left(\Na -2\right)}{\Na^{2}} \left(-\Na+1, \mathbf{1}_{\Na-1}^{T}\right)^{T}
\end{equation}
\end{corollary}
\begin{proof}
For $d = \Na-1$, $M = (\Na-2)\mathbf{1}_{\Na - 1}$ and $B = \left(\Na^{2}I + \left(\frac{2}{\Na} - \Na\right)\mathbf{1}_{\Na}\mathbf{1}_{\Na}^{T}\right)^{-1}$. 
Applying the Sherman-Morrison formula \cite{sherman1950adjustment} results in  
\begin{equation*}
B = \frac{1}{\Na^{2}}I - \frac{1}{2\Na^{2}}\left(\frac{2}{\Na}-\Na\right)\mathbf{1}_{\Na}\mathbf{1}_{\Na}^{T}
\end{equation*}
Substituting d, $B$, $M$ in Eqn.~\eqref{eq:centrality_regular_graph_approx} gives Eqn.~\eqref{eq:centrality_complete_graph_approx}. 
\end{proof}

\begin{corollary}[Effective Centrality on Ring Graphs with One Heterogeneous Node]
\label{corollary:eig_vec_approx_circle_graph}
Let $\mathcal{G}$ in Theorem \ref{thm:centrality_regular_graph_approx} be a ring network and $\delta = \left( \hat{\mathbf{p}}_{1}^{T}.\hat{\mathbf{p}} -1\right)$. Here $d = 2$, and the simplified approximate eigenvector centrality is as follows:
\begin{itemize}
\item If $\Na=3$, $\mathbf{v}_{\max}' \approx \mathbf{1}_{\Na} - \frac{1}{9}\delta \left(-2,1,1\right)^{T}$
\item If $\Na=4$, $\mathbf{v}_{\max}' \approx \mathbf{1}_{\Na} - \frac{1}{2}\delta \left(-1,0,1,0\right)^{T}$
\item If $\Na\geq 5$, $\mathbf{v}_{\max}' \approx \mathbf{1}_{\Na} + \frac{1}{\Na}\delta \sum_{k=1}^{\Na-1}\cot^{2}\left(\frac{\pi k}{\Na}\right)\mathbf{v}_{k+1}$
where $\mathbf{v}_{k+1} = \left(1, \omega^{k}, \omega^{2k},\cdots,\omega^{\left(\Na-1\right)k}\right) \in \mathbb{C}^{\Na}$ and $\omega = \exp \left(\frac{2 \pi i}{\Na}\right) \in \mathbb{C}$.
\end{itemize}
\end{corollary}
For homogeneous agents in a regular graph, standard eigenvector centrality predicts that all agents are equally central. From Corollary \ref{corollary:eig_vec_approx_circle_graph}, we see that adding just one agent that is different from the rest redistributes effective influence, with centrality increasing monotonically with graph distance from the heterogeneous agent. We illustrate this for even and odd numbers of agents in Fig.~\ref{fig:Ring_network}.
\begin{figure}[!htbp]
    \centering
    \includegraphics[width=0.75\linewidth]{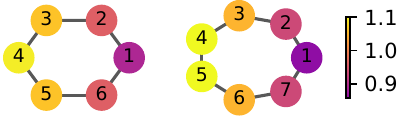}
    \caption{Approximate centrality redistribution for ring networks with $\Na = 6$ and $\Na = 7$. Here, $\delta = -0.1$ and node 1 is heterogeneous from the rest.}
    \label{fig:Ring_network}
\end{figure}

\section{Simulations}
\label{sec:Simulations}
We illustrate the intuition behind the results in Section \ref{sec:Projection-Constrained NOD} in simulation, with $S=\tanh$. 
We consider the six-agent network $\mathcal{G}_{1}$ in Fig. \ref{fig:Homogeneous constraints custom graph} in which all agents share homogeneous projection constraints and agents 2 and 5 are the most central, with eigenvector centrality approximately given by $(0.28,0.50,0.41,0.28,0.50,0.41)^{T}$. 
A non-zero bias $\mathbf{b}_{2} = (1,1,-1)^{T}$ is applied only at agent 2.  
On the homogeneous graph, agent 2 has a positive effective bias and skews the decision of the whole group towards the positive decision state, as shown via the unfolding diagram in Fig. \ref{fig:Positive_unfolding_custom_graph_siz_agents}. In  Fig. \ref{fig:Heterogenoeous constraitns custom graph}, agent 2 has a different projection constraint from the rest of the group, and its effective bias from the same input vector becomes negative, skewing the group decision towards the negative state. Simulation trajectories corresponding to these scenarios are shown in Fig. \ref{fig:Custom_graph_trajectories}. This illustrates how selective sensitivity to biases from local constraints can strongly change the emergent decision in the group.

In the first simulation example, the change in the collective decision was induced directly by the change in a constraint in the individual receiving task-relevant information. However, the agent decisions can also change even when the input is not applied to the agent whose projection constraint varies from the rest. This is due to changes in centrality in the induced weighted social graph.
In Fig. \ref{fig:Custom_graph_trajectories_2}, inputs are applied to agents 1 and 4, while the projection constraint of agent 2 is varied. When the constraint on agent 2 is introduced, the eigenvector centrality of the weighted network with heterogeneous projection constraints is $(0.25,0.41,0.44,0.36,0.54,0.39)^{T}$.
Observe that the centrality index of agent 1 decreases from the homogeneous case, while that of agent 4 increases. The two were equally central under homogeneous constraints. In both the homogeneous and heterogeneous simulations, agent 1 has effective bias, $b_{e1} = 0.52$, while agent 4 has effective bias $b_{e4} = -0.42$. However, due to the shift in centrality, despite having a weaker effective bias locally, agent 4 has more influence in the group decision once heterogeneity is introduced via agent 2, 
and the network decision skews negative.

\begin{figure}[!htbp]
    \centering
    
    \begin{subfigure}{0.38\linewidth}
        \centering
        \includegraphics[width=\linewidth]{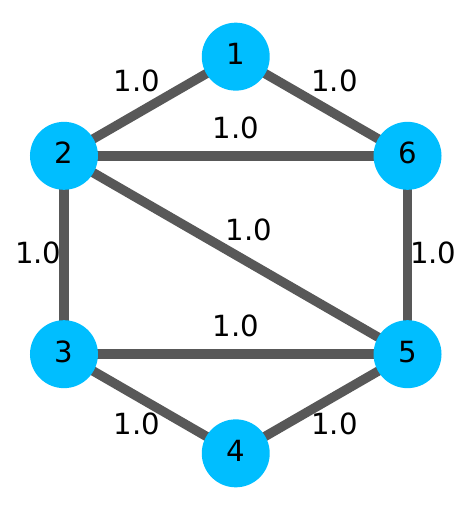}
        \caption{$\mathcal{G}_{1}$}
        \label{fig:Homogeneous constraints custom graph}
    \end{subfigure}
    \hfill
    \begin{subfigure}{0.58\linewidth}
        \centering
        \includegraphics[width=\linewidth]{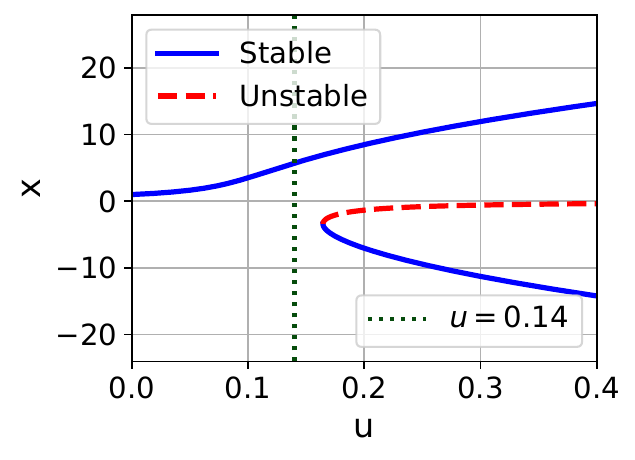}
        \caption{Positive unfolding}
        \label{fig:Positive_unfolding_custom_graph_siz_agents}
    \end{subfigure}
    \\
    \begin{subfigure}{0.38\linewidth}
        \centering
        \includegraphics[width=\linewidth]{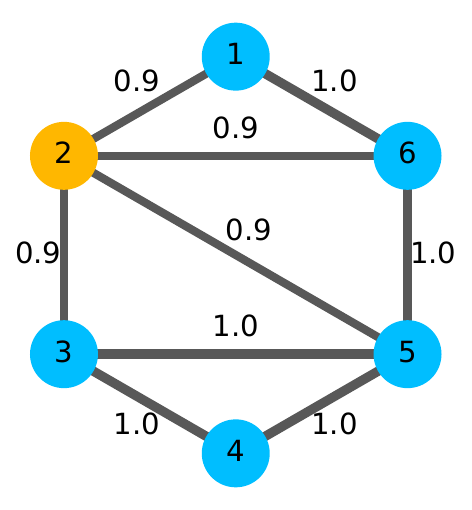}
        \caption{$\mathcal{G}_{2}$}
        \label{fig:Heterogenoeous constraitns custom graph}
    \end{subfigure}
    \hfill
    \begin{subfigure}{0.58\linewidth}
        \centering
        \includegraphics[width=\linewidth]{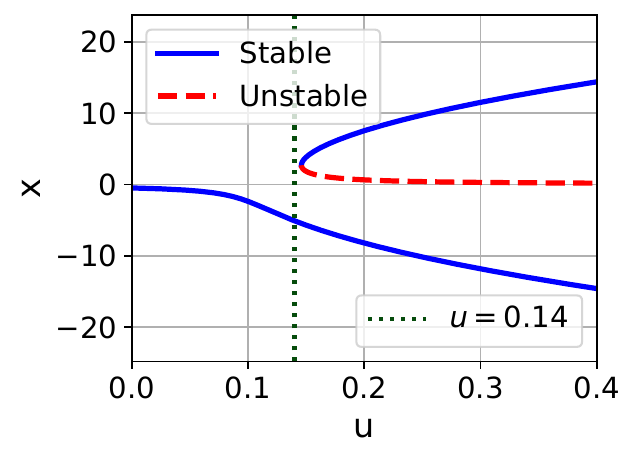}
        \caption{Negative unfolding}
        \label{fig:Negative_unfolding_custom_graph_siz_agents}
    \end{subfigure}
    
    \caption{Effective interaction graphs at bifurcation for identical network topologies with (a) homogeneous constraints $\mathbf{p} = \left(1,1,1\right)^{T}$ and (c) heterogeneous projection constraints where $\mathbf{p}_2 = \left(1,1,3\right)^{T}$ and $\mathbf{p}_i = \mathbf{p}$ for all other $i \in \{1, \dots, 6\} \setminus \{2\}$. Numerical edge labels are effective weights $\hat{\mathbf{p}}_i \cdot \hat{\mathbf{p}}_j$. Panels (b) and (d) show local bifurcation diagram topology of \eqref{eq:CNOD-multi-agnet-y-domain} over (a) and (c), respectively, with $\alpha = 1, d = 0.3, \gamma = 0.5, \mathbf{b}_{2} = (1,1,-1)^{T}$ and all other agents unbiased. Effective bias of agent 2 in (b) is $b_{e2} = 0.6$, unfolding the pitchfork bifurcation organizing the collective decision towards the positive branch; in (d) it is $b_{e2} = -0.3$, unfolding it towards the negative branch.
    }
    \label{fig:Custom_network_six_agents_simulation}
\end{figure}

\begin{figure}[!htbp]
    \centering
    \includegraphics[width=\linewidth]{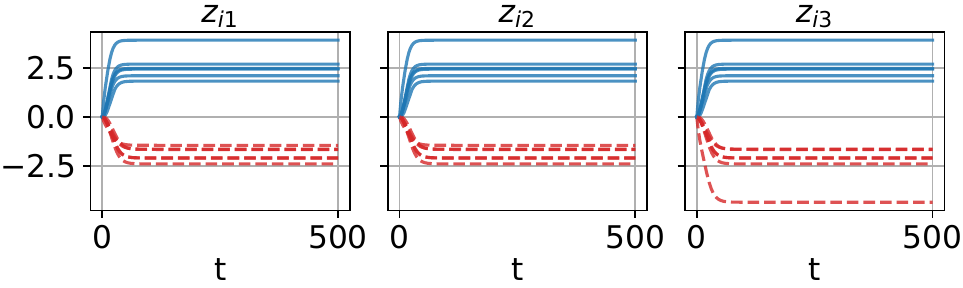}
    \caption{Opinion trajectories for networks $\mathcal{G}_{1}$ (solid blue lines) and $\mathcal{G}_{2}$ (dashed red lines) of Fig. \ref{fig:Custom_network_six_agents_simulation}. The parameters used are $u=0.14$, $\alpha = 1$, $d=0.3$ and $\gamma = 0.5$.}
    \label{fig:Custom_graph_trajectories}
\end{figure}

\begin{figure}[!htbp]
    \centering
    \includegraphics[width=\linewidth]{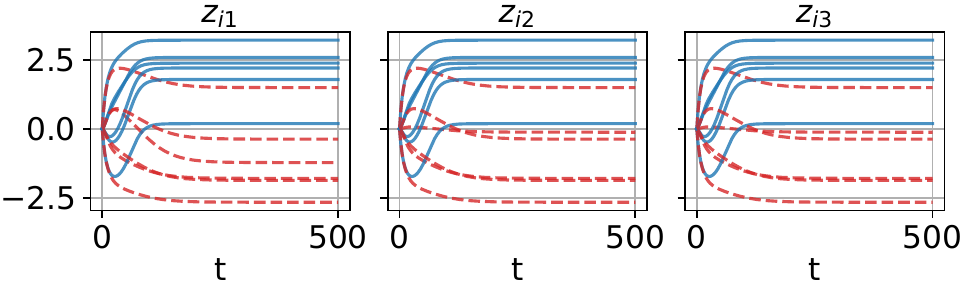}
    \caption{Opinion trajectories of \eqref{eq:CNOD-simple} over networks with homogeneous (solid blue) and heterogeneous (dashed red) projection constraints.  
    The homogeneous graph is $\mathcal{G}_1$ from Fig. \ref{fig:Custom_network_six_agents_simulation} with constraint $\mathbf{p} = \left(1,1,1\right)^{T}$ for all nodes. The heterogeneous graph is the same as $\mathcal{G}_2$ in Fig. \ref{fig:Custom_network_six_agents_simulation}, except $\mathbf{p}_2 = \left(1,0.1,0.1\right)^{T}$.
    The parameters used are $u=0.14$, $\alpha = 1$, $d=0.3$ and $\gamma = 0.5$. The non-zero biases are $\mathbf{b}_{1} = (0.3,0.3,0.3)^{T}$ and $\mathbf{b}_{4} = (-0.24,-0.24,-0.24)^{T}$ for both simulations. The effective bias in the collective decision is $\mathbf{v}^{T}\mathbf{b}_{e} = 0.03$ for $\mathcal{G}_1$, with the network selecting positive opinion configurations, and  $-0.02$ for the heterogeneous graph, skewing the decision in the negative direction.   
    }
    \label{fig:Custom_graph_trajectories_2}
\end{figure}

\section{Conclusion}
\label{sec:Conclusion}
We studied the effects of heterogeneous projection constraints on collective decision-making in the NOD framework, established conditions for a pitchfork bifurcation under rank 1 constraints, and showed how heterogeneous constraint alignments transform an unweighted effective network into an effective weighted network and redistribute effective centrality of agents in the collective decision. The presented results characterize selectively sensitive decision-making near a bifurcation: projected inputs bias the emerging collective decision at leading order through their alignment with the critical network mode, which itself is shaped by the pairwise constraint alignment between nodes on the graph. This provides a geometric analogue of selective sensitivity in neural decision-making, where it has been observed that neural circuits can receive information about several features of a stimulus while only selectively using those that are relevant to the current task \cite{mante2013context}. Future work will extend the analysis to higher-order and time-varying constraints. In particular, higher-dimensional rank constraints lead to larger reduced systems and may admit higher-dimensional critical eigenspaces and higher codimension bifurcation phenomena. Distinguishing between classes of constraints that lead to distinct bifurcation phenomena and characterizing their sensitivity to distributed input is open for investigation. 

\bibliographystyle{IEEEtran}
\bibliography{references}

\begin{thebibliography}{10}
\providecommand{\url}[1]{#1}
\csname url@samestyle\endcsname
\providecommand{\newblock}{\relax}
\providecommand{\bibinfo}[2]{#2}
\providecommand{\BIBentrySTDinterwordspacing}{\spaceskip=0pt\relax}
\providecommand{\BIBentryALTinterwordstretchfactor}{4}
\providecommand{\BIBentryALTinterwordspacing}{\spaceskip=\fontdimen2\font plus
\BIBentryALTinterwordstretchfactor\fontdimen3\font minus \fontdimen4\font\relax}
\providecommand{\BIBforeignlanguage}[2]{{%
\expandafter\ifx\csname l@#1\endcsname\relax
\typeout{** WARNING: IEEEtran.bst: No hyphenation pattern has been}%
\typeout{** loaded for the language `#1'. Using the pattern for}%
\typeout{** the default language instead.}%
\else
\language=\csname l@#1\endcsname
\fi
#2}}
\providecommand{\BIBdecl}{\relax}
\BIBdecl

\bibitem{ren2007distributed}
W.~Ren and E.~Atkins, ``Distributed multi-vehicle coordinated control via local information exchange,'' \emph{International Journal of Robust and Nonlinear Control}, vol.~17, no. 10-11, pp. 1002--1033, 2007.

\bibitem{rios2016automated}
J.~Rios-Torres and A.~A. Malikopoulos, ``Automated and cooperative vehicle merging at highway on-ramps,'' \emph{IEEE Transactions on Intelligent Transportation Systems}, vol.~18, no.~4, pp. 780--789, 2017.

\bibitem{gerkey2004formal}
B.~P. Gerkey and M.~J. Matarić, ``A formal analysis and taxonomy of task allocation in multi-robot systems,'' \emph{The International Journal of Robotics Research}, vol.~23, no.~9, pp. 939--954, 2004.

\bibitem{khamis2015multi}
A.~Khamis, A.~Hussein, and A.~Elmogy, \emph{Multi-robot Task Allocation: A Review of the State-of-the-Art}.\hskip 1em plus 0.5em minus 0.4em\relax Cham: Springer International Publishing, 2015, pp. 31--51.

\bibitem{degroot1974reaching}
M.~H. Degroot, ``Reaching a consensus,'' \emph{Journal of the American Statistical Association}, vol.~69, no. 345, pp. 118--121, 1974.

\bibitem{seeley1999group}
T.~D. Seeley and S.~C. Buhrman, ``Group decision making in swarms of honey bees,'' \emph{Behavioral Ecology and Sociobiology}, vol.~45, no.~1, pp. 19--31, 1999.

\bibitem{jia2015opinion}
P.~Jia, A.~MirTabatabaei, N.~E. Friedkin, and F.~Bullo, ``Opinion dynamics and the evolution of social power in influence networks,'' \emph{SIAM Review}, vol.~57, no.~3, pp. 367--397, 2015.

\bibitem{altafini2012consensus}
C.~Altafini, ``Consensus problems on networks with antagonistic interactions,'' \emph{IEEE Transactions on Automatic Control}, vol.~58, no.~4, pp. 935--946, 2013.

\bibitem{fontan2018achieving}
A.~Fontan and C.~Altafini, ``Achieving a decision in antagonistic multi agent networks: frustration determines commitment strength,'' in \emph{IEEE Conference on Decision and Control}, 2018, pp. 109--114.

\bibitem{he2023opinion}
G.~He, Z.~Shen, T.~Huang, W.~Zhang, and X.~Wu, ``Opinion dynamics with heterogeneous multiple interdependent topics on the signed social networks,'' \emph{IEEE Transactions on Systems, Man, and Cybernetics: Systems}, vol.~53, no.~10, pp. 6181--6193, 2023.

\bibitem{bizyaeva2023tac}
A.~Bizyaeva, A.~Franci, and N.~E. Leonard, ``Nonlinear opinion dynamics with tunable sensitivity,'' \emph{IEEE Transactions on Automatic Control}, vol.~68, no.~3, pp. 1415--1430, 2023.

\bibitem{bizyaeva2025tac}
------, ``Multitopic belief formation through bifurcations over signed social networks,'' \emph{IEEE Transactions on Automatic Control}, vol.~70, no.~8, pp. 5082--5097, 2025.

\bibitem{leonard2024fast}
N.~E. Leonard, A.~Bizyaeva, and A.~Franci, ``Fast and flexible multiagent decision-making,'' \emph{Annual Review of Control, Robotics, and Autonomous Systems}, vol.~7, pp. 19--45, 2024.

\bibitem{cathcart2023proactive}
C.~Cathcart, M.~Santos, S.~Park, and N.~E. Leonard, ``Proactive opinion-driven robot navigation around human movers,'' in \emph{IEEE/RSJ International Conference on Intelligent Robots and Systems}, 2023, pp. 4052--4058.

\bibitem{paine2024model}
T.~M. Paine and M.~R. Benjamin, ``A model for multi-agent autonomy that uses opinion dynamics and multi-objective behavior optimization,'' in \emph{2024 IEEE International Conference on Robotics and Automation (ICRA)}, 2024, pp. 8305--8311.

\bibitem{martinez2025avocado}
D.~Martinez-Baselga, E.~Sebastián, E.~Montijano, L.~Riazuelo, C.~Sagüés, and L.~Montano, ``Avocado: Adaptive optimal collision avoidance driven by opinion,'' \emph{IEEE Transactions on Robotics}, vol.~41, pp. 2495--2511, 2025.

\bibitem{qi2025integrating}
S.~Qi, Z.~Tang, Z.~Sun, and S.~Haesaert, ``Integrating opinion dynamics into safety control for decentralized airplane encounter resolution,'' in \emph{IEEE/RSJ International Conference on Intelligent Robots and Systems}, 2025, pp. 173--178.

\bibitem{mirtabatabaei2012opinion}
A.~Mirtabatabaei and F.~Bullo, ``Opinion dynamics in heterogeneous networks: Convergence conjectures and theorems,'' \emph{SIAM Journal on Control and Optimization}, vol.~50, no.~5, pp. 2763--2785, 2012.

\bibitem{liang2013opinion}
H.~Liang, Y.~Yang, and X.~Wang, ``Opinion dynamics in networks with heterogeneous confidence and influence,'' \emph{Physica A: Statistical Mechanics and its Applications}, vol. 392, no.~9, pp. 2248--2256, 2013.

\bibitem{li2023bounded}
G.~J. Li and M.~A. Porter, ``Bounded-confidence model of opinion dynamics with heterogeneous node-activity levels,'' \emph{Physical Review Research}, vol.~5, p. 023179, Jun 2023.

\bibitem{zhang2026modeling}
X.~Zhang, Q.~Wang, F.~Song, F.~Wang, and C.~Qu, ``Modeling and controlling polymorphic opinion dynamics in social networks with heterogeneous sensitivities,'' \emph{IEEE Transactions on Network Science and Engineering}, vol.~13, pp. 6140--6155, 2026.

\bibitem{nedic2010constrained}
A.~Nedic, A.~Ozdaglar, and P.~A. Parrilo, ``Constrained consensus and optimization in multi-agent networks,'' \emph{IEEE Transactions on Automatic Control}, vol.~55, no.~4, pp. 922--938, 2010.

\bibitem{wankhede2025multi}
P.~Wankhede, N.~Mandal, S.~Martínez, and P.~Tallapragada, ``Multi-topic projected opinion dynamics for resource allocation,'' in \emph{IEEE Conference on Decision and Control}, 2025, pp. 3469--3476.

\bibitem{zhang2021dissensus}
Z.~Zhang, S.~Al-Abri, and F.~Zhang, ``Dissensus algorithms for opinion dynamics on the sphere,'' in \emph{IEEE Conference on Decision and Control}, 2021, pp. 5988--5993.

\bibitem{zhang2022opinion}
------, ``Opinion dynamics on the sphere for stable consensus and stable bipartite dissensus,'' \emph{IFAC-PapersOnLine}, vol.~55, no.~13, pp. 288--293, 2022, 9th IFAC Conference on Networked Systems NECSYS 2022.

\bibitem{zhang2024opinion}
Z.~Zhang, S.~Chen, S.~Mayberry, and F.~Zhang, ``Opinion-based strategy for distributed multi-robot task allocation in swarms of robots,'' in \emph{IEEE/RSJ International Conference on Intelligent Robots and Systems}, 2024, pp. 3476--3481.

\bibitem{zhang2025mixed}
Z.~Zhang, Y.~Li, S.~Al-Abri, and F.~Zhang, ``Mixed opinion dynamics on the unit sphere for multi-agent systems in social networks,'' in \emph{American Control Conference}, 2025, pp. 4824--4829.

\bibitem{bullo2018lectures}
F.~Bullo, \emph{Lectures on Network Systems}, 2018.

\bibitem{fox1968rates}
R.~L. Fox and M.~P. Kapoor, ``Rates of change of eigenvalues and eigenvectors.'' \emph{AIAA Journal}, vol.~6, no.~12, pp. 2426--2429, 1968.

\bibitem{bizyaeva2021control}
A.~Bizyaeva, T.~Sorochkin, A.~Franci, and N.~Ehrich~Leonard, ``Control of agreement and disagreement cascades with distributed inputs,'' in \emph{IEEE Conference on Decision and Control}, 2021, pp. 4994--4999.

\bibitem{schaeffer1985singularities}
M.~Golubitsky and D.~G. Schaeffer, \emph{Singularities and Groups in Bifurcation Theory}.\hskip 1em plus 0.5em minus 0.4em\relax Springer, 1985.

\bibitem{Hiesmayr2024}
E.~V. Hiesmayr, ``Interactions in random structures,'' Ph.D. dissertation, University of California, Berkeley, 2024.

\bibitem{sherman1950adjustment}
J.~Sherman and W.~J. Morrison, ``Adjustment of an inverse matrix corresponding to a change in one element of a given matrix,'' \emph{The Annals of Mathematical Statistics}, vol.~21, no.~1, pp. 124--127, 1950.

\bibitem{mante2013context}
V.~Mante, D.~Sussillo, K.~V. Shenoy, and W.~T. Newsome, ``Context-dependent computation by recurrent dynamics in prefrontal cortex,'' \emph{nature}, vol. 503, no. 7474, pp. 78--84, 2013.

\end{thebibliography}

\end{document}